\documentclass[reqno]{amsart}
\begin{document}
\title[%\hfilneg AML-2016/14\hfil 
 Inverse source problem for advection-diffusion]
{A novel method to solve inverse source problem for advection-diffusion equation from final data}

\author[Z. Li %\hfil AML-2016/14\hfilneg
]
{Zhiyuan Li, \quad Gongsheng Li, \quad Xiangzheng Jia}

\address{Zhiyuan Li, \quad Gongsheng Li, \quad Xianzheng Jia \newline
School of Mathematics and Statistics,
Shandong University of Technology.
266 Xincunxi Road, Zibo, Shandong 255049, China}
\email{zyli@sdut.edu.cn,\ ligs@sdut.edu.cn,\ jxz\_1@sdut.edu.cn}

\thanks{Submitted \today.}
\keywords{advection-diffusion; inverse problem; final observation; uniqueness; \hfill\break\indent  Liouville's theorem.
}

\begin{abstract}
In this article, for an advection-diffusion equation we study an inverse problem for restoration of source temperature from the information of final temperature profile. The uniqueness of this inverse problem is established by taking an integral transform and using Liouville's theorem (complex analysis).
\end{abstract}

\maketitle
\numberwithin{equation}{section}
\newtheorem{lemma}{Lemma}[section]
\newtheorem{theorem}{Theorem}
\allowdisplaybreaks

\section{Introduction and main results}
Throughout this paper, assuming $T$ is a given positive number, and $\Omega\subset\mathbb R^d$ is a bounded domain with a sufficiently smooth boundary $\partial\Omega$, which is defined e.g., by some $C^2$ functional relations, we study the following initial-boundary value problem
\begin{equation} \label{eq}
\left\{
\begin{alignedat}{2}
&\partial_t u - L u = \rho(t)f(x), &\quad&  (x,t) \in Q, \\
&u(x,0)=0,&\quad& x\in \Omega, \\
&u(x,t)=0, &\quad&  (x,t)\in \partial\Omega\times(0,T],
\end{alignedat}
\right.
\end{equation} 
where $Q:=\Omega\times(0,T]$, $\rho$ is known as a $C^1[0,T]$-function, and $L$ is a second-order linear differential operator from $H^2(\Omega)\cap H_0^1(\Omega)$ to $L^2(\Omega)$ which can be written in the following form
$$
Lu:= \sum_{i,j=1}^d \partial_i(a_{ij}(x)\partial_j u) + \sum_{j=1}^d b_j(x)\partial_j u + c(x)u,\quad u\in H^2(\Omega)\cap H_0^1(\Omega),
$$
where $c,b_j\in L^\infty(\Omega)$, $j=1,\cdots,d$, and $a_{ij}=a_{ji}\in C^1(\overline\Omega)$, $1\le i,j\le d$ are given functions. Moreover there exists a positive constant $a_0 > 0$ such that
$$
\sum_{i,j=1}^d a_{ij}(x) \xi_i\xi_j \ge a_0 |\xi|^2,\quad \forall x\in\overline\Omega,\ \forall\xi:=(\xi_1,\cdots,\xi_d)\in\mathbb R^d.
$$

The differential equation in \eqref{eq} is called an advection-diffusion equation and is used as model equation for the description of density fluctuations in a material undergoing diffusion. It is also used to describe processes exhibiting diffusive-like behavior, for instance the \lq diffusion' of alleles in a population in population genetics (see, e.g., \cite{AW75}). We also refer to several fields of applications such as hydrology \cite{B72}, chemistry \cite{BN91}, and also in spatial ecology to modelize the dynamic of a population \cite{M02}. There are also many works from theoretical aspects and here we do not intend to give a complete list of references, and the readers can consult Chapter 7 in \cite{Evans} and \cite{F08} for example.  

In \eqref{eq}, the term $\rho(t)f(x)$ models a source term, and it is crucial to know or determine the coefficient $f$. For example, \cite{CC04} pointed out that in population dynamic, the extinction of the population, its persistence and its evolution is conditioned by this term $f$. However, usually this term cannot be directly measured due to the mixing of the effects of several factors, which requires
one to use inverse problems to identify these quantities by involving additional information that can be observed or measured practically. In other words, our main concern is:

{\bf Inverse source problem.} Let us assume that $\rho$ is known. Determine the unknown source magnitude $f$ in $\Omega$ from the final state observation 
$$
u(x,T),\quad \mbox{$x\in\Omega$.}
$$

Similar kinds of inverse problems have been studied by many authors. For $b_j\equiv0$, $j=1,\cdots,d$, \cite{HS13} and \cite{R80} established a uniqueness of the inverse problem in determining the spacewise-dependent source by the use of semi-group theory. For a more general source term $\rho$ which is $x$- and $t$-dependent, as is seen from Theorem 3.1 in \cite{I91}, the assumption $\rho(x,t)>0$ in $\overline Q$ is not sufficient to guarantee the uniqueness of the inverse problem. Actually, in this case, we need to assume $\partial_t \rho(x,t)>0$ in $\overline Q$. However, under a further assumption $\partial_t\rho>0$, \cite{I91} employed the strong maximum principles for the elliptic and parabolic equations to derive the uniqueness. There are some other tricks to overcome the ill-posedness of this inverse problem, we refer to \cite{CY96} and \cite{CY99} where the diffusion equation was parametrized by a diffusion coefficient and a stability inequality was established except for a countable set of parameters provided that $\rho(\cdot,T)>0$ in $\overline\Omega$, and refer to \cite{BC16} in which the Carleman estimate was used to recover the source from the measurement of the temperature at fixed time provided that the source is known in an arbitrary subdomain.

In this paper, for the advection-diffusion equation in \eqref{eq}, we show that in the case where $\rho$ is only $t$-dependent, the assumption $\rho>0$ in $[0,T]$ is enough to determine the unknown source term from the final overdetermination. We have

\begin{theorem}
\label{thm-unique}
Assume $\rho\in C^1[0,T]$ is positive for any $t\in[0,T]$. Let $f\in L^2(\Omega)$ and suppose $u$ solves the initial-boundary value problem \eqref{eq}. Then $u(\cdot,T)\equiv0$ in $\Omega$ implies $f\equiv0$ in $\Omega$.
\end{theorem}

Here we should mention that in the case where $b_j=0$ and $c\ge0$, we can derive the uniqueness and stability results in the framework of the Fourier expansion. Indeed, by using the orthonormality of the Dirichlet eigensystem of elliptic operator $L$, the unknown source term $f$ can be represented by the final measurement. It will be difficult to follow this argument to derive the uniqueness in the general case. On the other hand, it is also very difficult to follow the semigroup arguments used in \cite{HS13} and \cite{R80} because of the nonsymmetric term. 

The key idea of our proof is to transfer the advection-diffusion equation to an elliptic one by using an integral transform with respect to $t$. Several technical lemmas for this elliptic equation are needed, so we collect them in Section \ref{sec-unique}. Preparing all necessities, say, Lemmas \ref{lem-rho} and \ref{lem-analytic}, by Liouville's theorem (complex analysis) we will finish the proof of Theorem \ref{thm-unique}. Finally, concluding remarks are given in Section \ref{sec-rem}.

%%%%%%%%%%%%%%%%%%%%%%%%%%%%%%%%%%%%%%%%%%%%%
%%%%%%%%%%%%%%%%%%%%%%%%%%%%%%%%%%%%%%%%%%%%%
\section{Proof of Theorem \ref{thm-unique}}\label{sec-unique}

In this section, we will give a proof of our main result. For this, we multiply both sides of the differential equation in \eqref{eq} by $e^{-(s+M)t}$ with $s\in\mathbb C$ and $M>0$, and integrate from $0$ to $T$ to derive
\begin{equation}
\label{trans-u}
\int_0^T \partial_t u(t) e^{-(s+M)t} dt - \int_0^T L u(t) e^{-(s+M)t} dt = f(x)\int_0^T \rho(t) e^{-(s+M)t} dt.
\end{equation}
From integration by parts, and noting that $u(0)=u(T)=0$, we further calculate the first term on the left hand side of the above identity as follows
$$
\int_0^T \partial_t u(t) e^{-(s+M)t} dt
=(s+M)\int_0^T  u(t) e^{-(s+M)t} dt.
$$
For short, we put 
\begin{equation}
\label{eq-trans-u}
\widehat u(s):=\int_0^T  u(t) e^{-(s+M)t} dt,\quad \widehat\rho(s):=\int_0^T\rho(t) e^{-(s+M)t} dt.
\end{equation}
Therefore \eqref{trans-u} can be rephrased by
\begin{equation}
\label{eq-lap-u}
- L \widehat u(s) + (s+M)\widehat u(s) = f(x)\widehat\rho(s),\quad s\in\mathbb C.
\end{equation}
Next we are to give an estimate for $\widehat u(s)$ for $s\in\mathbb C$ on the basis of the above equation. We first introduce a useful lemma which gives a lower bound of the function $|\widehat \rho(s)|$ for some $s\in\mathbb C$.
\begin{lemma}
\label{lem-rho}
Let $N>1$ be a sufficiently large and fixed constant. Assuming the function $\rho\in C^1[0,T]$ is positive in $[0,T]$. Then there exists a positive constant $C=C(\inf_{[0,T]} \rho=:\rho_0,T,N)$ such that the following inequality  
$$
|\widehat \rho(s)| = \left|\int_0^T \rho(t) e^{-st} dt \right| \ge C\int_0^T e^{-t\Re s} dt
$$ 
is valid for any $s\in\{s\in\mathbb C;\ \Re s\le 0,\ \Im s\ge0,\ N\Im s+\Re s\le0\}$.
\end{lemma}
\begin{proof}
Firstly, by integration by parts we represent $\widehat \rho(s)$ as follows
$$
\int_0^T \rho(t) e^{-st} dt
= \frac{\rho(0) - \rho(T) e^{-sT} }{s}
+ s^{-1} \int_0^T \rho'(t) e^{-st} dt
=:I_1(s) + I_2(s).
$$
For $I_2(s)$, by a direct calculation, we find
$$
|I_2(s)| 
\le |s|^{-1} \int_0^T |\rho'(t)| e^{-t\Re s} dt
\le \frac{\|\rho\|_{C^1[0,T]}}{|s|}  \frac{1 - e^{-T\Re s}}{\Re s}.
$$
Moreover, the choice of $s$, say, $\Re s\le0$, $\Im s\ge0$ and $N\Im s+\Re s\le0$ imply that $N\Im s\le |\Re s|$, hence that 
\begin{equation}
\label{esti-s-Res}
|s| = \sqrt{(\Re s)^2 + (\Im s)^2} \le |\Re s| + |\Im s| \le (1+1/N) |\Re s|
\le 2 |\Re s|,\quad N>1,
\end{equation}
which can be further used to derive that
$$
|I_2(s)| 
\le \frac{2\|\rho\|_{C^1[0,T]}( e^{T|\Re s|} - 1)}{|s|^2}
\le \frac{4\|\rho\|_{C^1[0,T]} e^{T|\Re s|}}{|s|^2}.
$$
Let us turn to evaluating $I_1(s)$. For this, a direct calculation derives
$$
sI_1(s) = \rho(0) - \rho(T) e^{-T\Re s} \cos(T\Im s)  + i \rho(T) e^{-T\Re s} \sin(T\Im s), 
$$
hence that
$$
|sI_1(s)|^2 = \rho(T)^2 e^{-2T\Re s} -2\rho(0)\rho(T) e^{-T\Re s} \cos(T\Im s) + \rho(0)^2, 
$$
finally that
$$
|I_1(s)|^2 = \frac{\left(\rho(T) e^{-T\Re s} - \rho(0)\right)^2}{|s|^2} + 2\rho(0) \rho(T) e^{-T\Re s} \frac{1- \cos(T\Im s)}{|s|^2}.
$$
Moreover, by the use of the inequality
$$
\frac{|1-\cos t|}{t^2} \le C,\quad t\in\mathbb R,
$$
we see that
$$
\frac{|1- \cos(T\Im s)|}{|s|^2}
=\frac{|1- \cos(T\Im s)|}{|T\Im s|^2} \frac{|T\Im s|^2}{|s|^2} \le CT^2,\quad s\in\mathbb C.
$$
Combining all the above estimates, we find
$$
|I_1(s)|^2 \ge \frac{\left(\rho(T) e^{-T\Re s} - \rho(0)\right)^2}{|s|^2} - 2CT^2\rho(0) \rho(T) e^{-T\Re s}.
$$
Now for sufficiently large $|s|$, then $|\Re s|$ is also sufficiently large in view of \eqref{esti-s-Res}. Consequently there exists a constant $C=C(T,\rho)$ such the estimate
$$
|I_1(s)| \ge \frac{Ce^{-T\Re s}}{|s|}.
$$
privided that $|s|>K$ with $K>0$ is a sufficiently large constant. Finally, we have
$$
\left|\int_0^T \rho(t) e^{-st} dt\right|
\ge |I_1(s)| - |I_2(s)|
\ge \frac{e^{T|\Re s|}}{|s|} \left(C - \frac{4\|\rho\|_{C^1[0,T]}  }{|s|} \right)
\ge \frac{Ce^{T|\Re s|}}{2|s|} 
$$
for $|s|>K$ satisfying $\Re s\le0$, $\Im s\ge0$ and $N\Im s + \Re s\le0$, where $N\ge1$.

Next, we evaluate $\widehat \rho(s)$ in the case when $s$ lies in $\{s\in\mathbb C; |s|\le K,\ \Re s\le0,\ \Im s\ge0, N\Im s + \Re s\le0\}$. By taking sufficiently large $N$, we are led to 
$$
|\Im s| \le \frac{|\Re s|}{N} \le \frac{K}{N},
$$
which implies that 
$$
|\sin (t\Im s)| \le |t\Im s| \le \frac{TK}{N} \le \frac14,
$$
and 
$$
\cos (t\Im s) \ge \frac{1}2,
$$

We first rewrite $\widehat \rho(s)$ as follows
$$
\widehat\rho(s) = \int_0^T \rho(t) e^{-t\Re s} \cos (t\Im s) dt + i \int_0^T \rho(t) e^{-t\Re s} \sin (t\Im s) dt,
$$
hence that
$$
|\widehat\rho(s)| \ge \frac12 \int_0^T \rho(t) e^{t|\Re s|}  dt - \frac14 \int_0^T \rho(t) e^{t|\Re s|} dt
\ge\frac{\rho_0}4 \int_0^T e^{t|\Re s|} dt
= \frac{\rho_0(e^{T|\Re s|}- 1)}{4|\Re s|},
$$
where in the last inequality we used the assumption $\rm A1$.

Finally, we obtain 
$$
|\widehat \rho(s)| \ge C\frac{e^{T|\Re s|} - 1}{|\Re s|}>C_1>0,\quad s\in\{s\in\mathbb C;\Re s\le0, \Im s\ge0, N\Im s+ \Re s\le0\},
$$
where the constant $C_1>0$ is only dependent on $T,\rho,N$.
\end{proof}

\begin{lemma}
\label{lem-analytic}
Under the assumptions in Theorem \ref{thm-unique}, the function $\widehat u: \mathbb C\to H^2(\Omega)\cap H_0^1(\Omega)$ which is defined by \eqref{eq-trans-u} can be estimated as follows
\begin{equation}
\label{esti-u}
\|\widehat u(s)\|_{L^2(\Omega)} 
\le C|\widehat \rho(s)| \|f\|_{L^2(\Omega)},\quad s\in \mathbb C.
\end{equation}
\end{lemma}
\begin{proof}
We denote $\overline{\widehat u}(s)$ as the conjugate of the complex valued function $\widehat u(s)$. Then multiplying $\overline{\widehat u}(s)$ on both sides of the equation \eqref{eq-lap-u}, and taking integral on $\Omega$ lead to
$$
\int_\Omega (-L\widehat u(s)) \overline{\widehat u}(s) + (s+M)\widehat u(s) \overline{\widehat u}(s) dx
= \int_\Omega \widehat \rho(s) \overline{\widehat u}(s) dx,\quad 
s\in\mathbb C.
$$
Here we will omit the variable $x$ in $\widehat u(x,s)$ when no confusion can arise. Integration by parts yields
\begin{equation}
\label{eq-var}
B_s[\widehat u(s),\widehat u(s)] = \int_\Omega \widehat \rho(s) \overline{\widehat u}(s) dx,
\end{equation}
where 
$$
B_s[\widehat u(s),\widehat u(s)]
=\int_\Omega \sum_{i,j=1}^d a_{ij} (\partial_j \widehat u(s)) \partial_i \overline{\widehat u}(s)
+ \sum_{j=1}^d b_j (\partial_j \widehat u(s)) \overline{\widehat u}(s) + (s+c+M) |\widehat u(s)|^2 dx.
$$
Then the symmetry of $\{a_{ij}\}$ implies that $\sum_{i,j=1}^d a_{ij} \partial_j \widehat u(s) \partial_i \overline{\widehat u}(s) \in\mathbb R$ for any $s\in\mathbb C$, hence the real part of $B_s[\widehat u,\widehat u] $ can be represented as follows
\begin{equation}
\label{eq-re}
\Re\left(B_s[\widehat u,\widehat u] \right)
=\int_\Omega \sum_{i,j=1}^d a_{ij} (\partial_j \widehat u) \partial_i \overline{\widehat u}
+ \Re\left[\sum_{j=1}^d b_j (\partial_j \widehat u) \overline{\widehat u}\right] + (\Re s+c+M) |\widehat u|^2 dx,
\end{equation}
while the imaginary part is
\begin{equation}
\label{eq-im}
\Im\left(B_s[\widehat u,\widehat u] \right)
=\int_\Omega\Im\left[\sum_{j=1}^d b_j (\partial_j \widehat u) \overline{\widehat u}\right] + (\Im s) |\widehat u|^2 dx.
\end{equation}
Now in the case of $\Re s\ge0$, from \eqref{eq-re}, we find
\begin{equation}
\label{esti-re}
\int_\Omega \sum_{i,j=1}^d a_{ij} (\partial_j \widehat u) \partial_i \overline{\widehat u}
+ (\Re s+c+M) |\widehat u|^2 dx
\le |B_s[\widehat u,\widehat u] | + \left|\sum_{j=1}^d b_j (\partial_j \widehat u) \overline{\widehat u} \right|,
\end{equation}
where we used the inequality $\Re s\le |\Re s|\le |s|$, if $s\in\mathbb C$. Since $c\in L^\infty(\Omega)$, we can choose $M>0$ large enough so that $\Re s+ c + M\ge M/2$ for $\Re s\ge0$. Moreover, the ellipticity of $\{a_{ij}\}$ and the use of Poincar\'e's inequality imply
$$
\int_\Omega \sum_{i,j=1}^d a_{ij} (\partial_j \widehat u) \partial_i \overline{\widehat u}
\ge a_0 \| \nabla \widehat u\|_{L^2(\Omega)}^2
\ge C\|\widehat u\|_{H^1(\Omega)}^2.
$$
Here and henceforth, $C>0$ is regarded as a generic constant which only depends on $d,\Omega, a_{ij},b_j,c,\rho$. Consequently, in view of \eqref{eq-var}, from the H\"{o}lder inequality we derive
\begin{align*}
&C\|\widehat u(s)\|_{H^1(\Omega)}^2
+ \frac{M}2 \|\widehat u(s)\|_{L^2(\Omega)}^2
\\
\le& |\widehat \rho(s)| \|f\|_{L^2(\Omega)} \|\widehat u(s)\|_{L^2(\Omega)} 
+ C\|\nabla \widehat u(s)\|_{L^2(\Omega)} \|u(s)\|_{L^2(\Omega)},\quad \Re s\ge0,
\end{align*}
and hence
\begin{align*}
&C\|\widehat u(s)\|_{H^1(\Omega)}^2
+ \frac{M}2 \|\widehat u(s)\|_{L^2(\Omega)}^2
\\
\le& |\widehat \rho(s)| \|f\|_{L^2(\Omega)}\|\widehat u(s)\|_{L^2(\Omega)}
+ \varepsilon C^2\|\nabla \widehat u(s)\|_{L^2(\Omega)}^2 + \frac{1}{4\varepsilon} \|u(s)\|_{L^2(\Omega)}^2,\quad \Re s\ge0.
\end{align*}
Now we choose $\varepsilon>0$ and $M>0$ such that 
$$
\varepsilon C^2 \le \frac{C}2,\quad \frac{M}2\ge \frac{1}{4\varepsilon},
$$
which implies that
\begin{align*}
\|\widehat u(s)\|_{H^1(\Omega)}
\le \frac{2}{C}|\widehat \rho(s)| \|f\|_{L^2(\Omega)},\quad \Re s\ge0.
\end{align*}
Next, if $s\in \{\Re s\le0; N\Im s + \Re s\ge0\}$, where $N>0$, again by the fact \eqref{eq-var} and the H\"older inequality, we evaluate the imaginary part of $B_s[\widehat u(s),\widehat u(s)]$ as follows
\begin{equation}
\label{esti-im}
(\Im s)  \int_\Omega |\widehat u(s)|^2 dx
\le  |\widehat \rho(s)| \|f\|_{L^2(\Omega)} \|\widehat u(s)\|_{L^2(\Omega)}
+ C\|\nabla\widehat u(s)\|_{L^2(\Omega)} \|\widehat u(s)\|_{L^2(\Omega)},
\end{equation}
Similarly, \eqref{eq-re} can be further treated by
\begin{align*}
&C\|\widehat u(s)\|_{H^1(\Omega)}^2
+ (\Re s+c+M) \|\widehat u(s)\|_{L^2(\Omega)}^2
\\
\le& |\widehat \rho(s)| \|f\|_{L^2(\Omega)} \|\widehat u(s)\|_{L^2(\Omega)}
+ C\|\nabla\widehat u(s)\|_{L^2(\Omega)} \|\widehat u(s)\|_{L^2(\Omega)},
\end{align*}
Collecting all the above estimates, it follows that
\begin{align*}
&C\|\widehat u(s)\|_{H^1(\Omega)}^2
+ (N\Im s + \Re s+c+M) \|\widehat u(s)\|_{L^2(\Omega)}^2
\\
\le& (N+1)|\widehat \rho(s)| \|f\|_{L^2(\Omega)} \|\widehat u(s)\|_{L^2(\Omega)}
+ C(N+1)\|\nabla\widehat u(s)\|_{L^2(\Omega)} \|\widehat u(s)\|_{L^2(\Omega)}
\\
\le& (N+1)|\widehat \rho(s)| \|f\|_{L^2(\Omega)}\|\widehat u(s)\|_{L^2(\Omega)}
+ \varepsilon C^2(N+1)^2\|\nabla\widehat u(s)\|_{L^2(\Omega)}^2 + \frac{1}{4\varepsilon}\|\widehat u(s)\|_{L^2(\Omega)}
,
\end{align*}
Now from the choice of $s$, say, $s\in \{\Re s\le0; N\Im s + \Re s\ge0\}$, we see that $N\Im s + \Re s+c+M\ge\frac{M}2$ provided that $M>0$ is sufficiently large. Moreover, by choosing $\varepsilon>0$ such that $\varepsilon C^2(N+1)^2 \le \frac{C}2$ and $\frac{M}2\ge \frac1{4\varepsilon}$, we arrive at the following inequality
\begin{align*}
\frac{C}2\|\widehat u(s)\|_{H^1(\Omega)}^2
\le (N+1)|\widehat \rho(s)| \|f\|_{L^2(\Omega)}\|\widehat u(s)\|_{L^2(\Omega)},
\end{align*}
that is
\begin{align*}
\|\widehat u(s)\|_{H^1(\Omega)}
\le \frac{2(N+1)}{C}|\widehat \rho(s)| \|f\|_{L^2(\Omega)},\quad \Re s\le0,\ N\Im s+ \Re s\ge0, \ N>0.
\end{align*}

Let us turn to considering the case when $s\in\{\Re s\le0; \Im s\ge0; N\Im s+\Re s\le0\}$ with a sufficiently large $N>0$. From the previous estimate for the function $\rho$ in Lemma \ref{lem-rho}, we have
$$
\|\widehat u(s)\|_{H^1(\Omega)} 
\le \int_0^T \|u(t)\|_{H^1(\Omega)} e^{-t\Re s}dt
\le C|\widehat \rho(s)| \|f\|_{L^2(\Omega)}.
$$
Collecting all the above estimates, we finally obtain
$$
\|\widehat u(s)\|_{H^1(\Omega)}  \le C|\widehat \rho(s)| \|f\|_{L^2(\Omega)},\quad s\in\mathbb C,
$$
which completes the proof of the lemma.
\end{proof}

Now we are ready to give a proof of our main theorem.

\begin{proof}[\bf Proof of Theorem \ref{thm-unique}]
Since $\widehat u(s)$ and $\widehat \rho(s)$ are analytic in $\mathbb C$, and hence $\frac{\widehat u(s)}{\widehat \rho(s)}$ is a holomorphic function in $\mathbb C$ removing the zero points of $\widehat\rho(s)$. However, Lemma \ref{lem-analytic} and Riemann's theorem (see, e.g., \cite{R87}) guarantee that all the singularities of the function $\frac{\widehat u(s)}{\widehat \rho(s)}$ are removable, and then this function can be regarded as an analytic function in $\mathbb C$ and is bounded on $\mathbb C$. Thus the Liouville theorem (see, e.g., \cite{R87}) implies that the analytic function $\frac{\widehat u(s)} {\widehat \rho(s)}$ must be independent of $s$, that is, there exists a function $\phi(x)$ such that
$$
\phi(x) = \frac{\widehat u(x;s)} {\widehat \rho(s)},\quad s\in\mathbb C.
$$
We claim that $\phi$ must vanish in $\Omega$. Indeed, by using integration by parts, and noting that $u(0)=u(T)=0$, we find that 
$$
s\widehat u(s) = s\int_0^T u(t) e^{-st}dt 
=\int_0^T u'(t) e^{-st}dt,\quad s>0,
$$
hence that
$$
|s\widehat u(s)| \le \|u\|_{C^1[0,1]} \int_0^T e^{-st}dt
\le \frac{\|u\|_{C^1[0,1]}(1-e^{-sT})}{s}\to 0,\ \mbox{ as $s\to+\infty$.}
$$
Similarly
$$
s\widehat \rho(s) = \rho(0) - \rho(T)e^{-sT} + \int_0^T u'(t) e^{-st}dt 
\to \rho(0),\ \mbox{ as $s\to+\infty$.}
$$
Since $\phi$ is independent of $s$, these two estimates for $s\widehat u(s)$ and $s\widehat\rho(s)$ imply that $\phi$ is vanished in $\Omega$, finally that 
$$
\widehat u(x;s)\equiv0,\quad x\in\Omega,\ s\in\mathbb R
$$
in view of the fact 
$$
\widehat \rho(s) = \int_0^T \rho(t) e^{-st} dt > 0,\quad s\in\mathbb R,
$$
hence that
$$
\int_0^\infty \widetilde u(t) e^{-st} dt = 0,\quad s\in\mathbb R,
$$
where $\widetilde u(t):=u(t)$ if $t\in(0,T)$ and $\widetilde u$ vanishes outside of $(0,T)$, which implies $u(x,t)\equiv0$ in $\Omega\times(0,T)$ due to the uniqueness of the Laplace transform, and hence $f\equiv0$ in $\Omega$. This completes the proof of the main result.

\end{proof}

%%%%%%%%%%%%%%%%%%%%%%%%%%%%%%%%%%%%%%%%%%%%%%%%%
%%%%%%%%%%%%%%%%%%%%%%%%%%%%%%%%%%%%%%%%%%%%%%%%%
\section{Concluding remarks}
\label{sec-rem}

In this paper, the initial-boundary value problem \eqref{eq} where the source term takes the form of separated variables was considered. In Theorem \ref{thm-unique} we proved that the spatial component in the source term can be uniquely determined by the final observation data. The key idea is to employ an integral transform with respect to the temporal variable to transfer the advection-diffusion equation to the corresponding elliptic equation with a parameter. On the basis of several estimates for this elliptic equation (see, Lemmas \ref{lem-rho} and \ref{lem-analytic}), we finished the proof by using Liouville's theorem (complex analysis). It should be mentioned here that the proof of the above uniqueness result heavily relies on the setting that all the coefficients in the operator $L$ is $t$-independent. It would be interesting to investigate what happens about this inverse problem in the framework of $(x,t)$-dependent coefficients.

%%%%%%%%%%%%%%%%%%%%%%%%%%%%%%%%%%%%%%%%%%%%%%%%%
%%%%%%%%%%%%%%%%%%%%%%%%%%%%%%%%%%%%%%%%%%%%%%%%%
\section*{Acknowledgement}
The author thanks Grant-in-Aid for Research Activity Start-up 16H06712, Japan Society of the Promotion of Science (JSPS).

%%%%%%%%%%%%%%%%%%%%%%%%%%%%%%%%%%%%%%%%%%%%%%%
%%%%%%%%%%%%%%%%%%%%%%%%%%%%%%%%%%%%%%%%%%%%%%%

\end{document}